\documentclass[11pt]{amsart}

\usepackage{color}

\usepackage{tikz-cd}
\usepackage[all]{xy}
 \usepackage[latin1]{inputenc} 
\usepackage[english]{babel}

\usepackage[T1]{fontenc}
 \usepackage{lmodern}
\usepackage{mathrsfs}
\usepackage{graphicx,epsfig,amscd,graphics,xypic}
 \usepackage[all]{xy}
\usepackage{amsmath,amssymb}

\usepackage[margin=4cm]{geometry}

\usepackage{hyperref}
\hypersetup{
    bookmarks=true,         
    unicode=false,          
    pdftoolbar=true,        
    pdfmenubar=true,        
    pdffitwindow=false,     
    pdfstartview={FitH},    
    pdftitle={My title},    
    pdfauthor={Author},     
    pdfsubject={Subject},   
    pdfcreator={Creator},   
    pdfproducer={Producer}, 
    pdfkeywords={keyword1} {key2} {key3}, 
    pdfnewwindow=true,      
    colorlinks=false,       
    linkcolor=red,          
    citecolor=green,        
    filecolor=magenta,      
    urlcolor=cyan           
}

\newcommand{\A}{\mathcal{A}}
\newcommand{\sk}{\smallskip}

\newcommand{\Aff}{\mathrm{Aff}(\mathbb{C})}
\newcommand{\C}{\mathbb{C}}
\newcommand{\Z}{\mathbb{Z}}
\newcommand{\R}{\mathbb{R}}
\newcommand{\U}{\mathbb{U}}
\newcommand{\dev}{\mathrm{dev}}
\newcommand{\Ht}{\mathrm{H}^1_{\alpha}(\Gamma,\mathbb{C}) }
\newcommand{\Hl}{\mathrm{H}^1( \Sigma, \mathbb{C}^*) }
\newcommand{\Li}{\mathrm{Li}}
\newcommand{\Tr}{\mathrm{Tr}}
\newcommand{\Mod}{\mathrm{Mod}(\Sigma)}
\newcommand{\car}{ \chi(\Gamma, \Aff)}
\newcommand{\Sp}{ \mathrm{Sp}(2g, \mathbb{Z})}
\newcommand{\Sr}{ \mathrm{Sp}(2g, \mathbb{R})}
\newcommand{\tor}{ \mathcal{I}(\Sigma)}
\newcommand{\Stab}{\mathrm{Stab}}

\newtheorem{thm}{Theorem}[section]

\newtheorem*{thm*}{Theorem}   
   
\newtheorem*{lemma*}{Lemma}

\newtheorem{prop}[thm]{Proposition}
\newtheorem{defi}{Definition}           
 
\newtheorem{lemma}[thm]{Lemma}
\newtheorem*{rem}{Remark}

\newtheorem*{ex*}{Example \ref{exJ2} (continued)}

\title[]{Mapping class group dynamics and the holonomy of branched affine structures.}

\author[Selim Ghazouani]{Ghazouani Selim}
\address{Département de Mathématiques et Applications, ENS Ulm, 45 rue d'Ulm 75 005 Paris {\it E-mail : selim.ghazouani@gmail.com} }

\begin{document}

\maketitle

\begin{abstract}

We classify, up to few exceptions, the orbit closures of the $\Mod$-action on the affine character variety $\chi( \Aff)$. We obtain from this classification that the only obstruction for a non-abelian  representation $\rho : \pi_1 \Sigma \longrightarrow \Aff$ to be the holonomy of a branched affine structure on $\Sigma$ is to be  Euclidean and not to have positive volume, where $\Sigma$ is a closed oriented surface of genus $g \geq 2$.

\end{abstract}

\section{Introduction.}

We introduce several notations that we are going to use throughout the text:

\begin{itemize}

\item $\Sigma$ is a closed oriented surface of genus $g \geq 2$.

\item $\Gamma$ is the fundamental group of $\Sigma$.

\item $\mathrm{Mod}(\Sigma)$ is the mapping class group of $\Sigma$.

\item $\mathrm{Aff}(\mathbb{C})$ is the complex affine group of dimension $1$.

\item $\chi(\Gamma, \Aff) = \chi$ is the character variety.

\end{itemize}

\subsection{Translation surfaces.}
A nowadays very popular topic, at the crossroads of dynamics, algebraic geometry and metric geometry, are translations surfaces. These are the structures whose local charts are given by locally integrating an holomorphic $1$-form on a Riemann surface. The periods of the associated $1$-form are geometric invariants of the structure, and they can be thought of as an element of

$$ \mathrm{Hom}(\mathrm{H}_1(\Sigma, \Z), \C) \simeq \mathrm{H}^1(\Sigma, \C). $$

\noindent In the language of geometric structures, this period morphism is the \textit{holonomy} of the translation structure. A very nice theorem, although relatively unknown, characterises the elements of $\mathrm{H}^1(\Sigma, \C) $ which can arise as the holonomy of a translation surface.

\begin{thm}[Haupt, \cite{Haupt}]
\label{Haupt}
A morphism $p \in \mathrm{H}^1(\Sigma, \C)$ can be realised as the holonomy of a translation surface if and only if the two following conditions hold:

\begin{enumerate}

\item the volume of $p$ is positive;

\item if the image of $p$ is a lattice $\Lambda$ in $\C$, then $\mathrm{vol}(p) > \mathrm{vol}(\C / \Lambda)$.

\end{enumerate}

\end{thm}
  
The volume $\mathrm{vol}$ of an element $p \in \mathrm{H}^1(\Sigma, \C)$ is the intersection product $\Re(p) \cdot \Im(p)$. The terminology 'volume' comes from the fact that a translation surface whose period is $p$ has volume equal to $\Re(p) \cdot \Im(p)$. This remark makes the first condition in the Theorem \ref{Haupt} obviously necessary.

\subsection{Complex projective structures.}

Another major matter of interest in the theory of surfaces are complex projective structures, whose model is $\mathbb{CP}^1$ with $\mathrm{PSL}(2, \C)$ acting through projective transformations. They historically arose when mathematicians of the $19^{th}$ century were studying particular cases of order two complex differential equations on Riemann surfaces. To each can be associated a complex projective structure whose holonomy is exactly the monodromy of the associated differential equation.   The question of determining the representations that can be realised has been solved by Gallo, Kapovich and Marden:

\begin{thm}[Gallo-Kapovich-Marden, \cite{GKM}]

A representation $\rho: \pi_1 \Sigma \longrightarrow \mathrm{PSL}(2,\C)$ is the holonomy of a projective structure if and only if the two following condition hold:

\begin{enumerate}

\item the image of $\rho$ is a non-elementary subgroup of $\mathrm{PSL}(2,\C)$;

\item $\rho$ can be lifted to $\mathrm{SL}(2,\C)$.

\end{enumerate}

\end{thm}

We say a projective structure is \textit{'branched'} when it has a finite number of special points which have punctured neighbourhoods projectively equivalent to a ramified cover of a point in $\mathbb{CP}^1$. For example a translation surface can be thought of as a particular case of \textit{branched} projective structure, thinking of $\C$ as a subset of $\mathbb{CP}^1$ and the set of translations as a subgroup of $\mathrm{PSL}(2,\C)$. Its branched points are the conical points, which are the zeros of the underlying holomorphic $1$-form.

\subsection{Branched affine structures.}

We will be interested in this article to \textit{branched affine structures}.
They lie somewhere inbetween translation and strictly projective structures: the model space is $\C$ and the transformation group is the one-dimensional complex affine group $\Aff$. A simple way to build examples of such structures is to glue the sides of a Euclidean polygon using affine transformations; the branched points will be located at the vertices. These are particular cases of branched projective structures as translation surfaces are particular cases of branched affine ones. 
\noindent Although these objects make their appearance in several different works (far from claiming for exhaustivity, see \cite{Mandelbaum}, \cite{Mandelbaum2},\cite{Gunning}, \cite{Veech}, \cite{McMullen}), they have not yet been investigated as a proper research subject. The author believes that these structures are rich and provide questions of both geometric and dynamic nature, together with natural links to  Teichmüller theory that make of them a distinguishable matter of interest. 

\vspace{2mm} This article will deal with the elementary question of determining which representations are the holonomy of a branched affine structure. 

\subsection{Mapping class group dynamics.}

A powerful tool to investigate the holonomy problem is the mapping class group action on the associated character variety. The latter parametrises the set of all representations of the fundamental group of a compact surface into the affine group $\Aff$ (up to conjugation), and is equipped with a natural action of the mapping class group by precomposition. 
\noindent A classical argument of Ehresmann popularised by Thurston ensures that the set of geometric representations is open in the character variety, and it is obviously invariant by the action of the mapping class group. A good understanding of this action is a path to the answer to the holonomy problem. In \cite{Ghazouani} we proved that the mapping class group action on the character variety is ergodic relatively to the Lebesgue measure. However, it only gives that almost every representation is geometric and does not characterise the obstruction to be.

\noindent In a remarkable unpublished note, Kapovich revisits Haupt's theorem recalled above (Theorem \ref{Haupt}), giving a proof completely based on the analysis of mapping class group action. In that case the action is nothing but the $\Sp$ action on an homogeneous space and the use of the powerful theorem of Ratner leads to a complete classification of the closed invariant subset of the character variety and consequently reproves Haupt's theorem.  
\noindent Note that mapping class group dynamics have already been extensively studied on $G$-character varieties, when $G$ is a reductive Lie group as $\mathrm{PSL}(2,\R)$,  $\mathrm{PSL}(2,\C)$ or $\mathrm{SU}(2)$, see for example \cite{Gold}, \cite{Gold2}, \cite{MW}, \cite{Palesi}, and for the case $G= \Aff$ we would like to mention \cite{CousinMoussard}.

\subsection{Results}

 In this article, we give a systematic description of the closure of the $\Mod$-orbit of a point in $\chi$. We are able to identify dynamical phenomenons of 'Ratner' type: the closure of an orbit is, up to few exceptions, a submanifold of $\chi$. More precisely, we prove the following theorem:

\begin{thm}

Consider $[\rho] \in \chi(\Gamma, \Aff)$ such that the image of its linear part $\alpha_{\rho}$ is not the group of $n^{th}$ roots of the unity for $n = 2,3,4,6$. Then the closure of its orbit under the $\Mod$-action is a real analytic submanifold of $\chi$. 

\end{thm}

From the precise description of the orbits that we give (which is a consequence of the classification results of Section \ref{orbitclosure} and of Theorem \ref{image}), and additional geometric constructions to handle the particular cases that are out of reach through the $\Mod$ approach, we are able to completely characterise the representations which are the holonomy of a branched affine structure:

\begin{thm}

Let $\rho : \Gamma \longrightarrow \Aff$ be a non-abelian representation.

\begin{itemize}

\item If $\rho$ is not Euclidean, then it is the holonomy of a branched affine structure.

\item If $\rho$ is Euclidean, it is the holonomy of a branched affine structure if and only if its volume is positive. 

\end{itemize}

\end{thm}

\subsection{Acknowledgements.}

We are very grateful to Julien Marché for having shared his infinite wisdom on twisted cohomology by which he brought a conceptual light on a calculation of the author, rendering the latter far less mysterious, and to Yves Benoist and Tyakal Venkataramana for valuable discussions about Lemma \ref{racine}. It is also a pleasure to thank Jeremy Daniel for having, although very reluctantly, answered several of our questions on Lie group theory.  Finally, the author is infinitely indebted to Bertrand Deroin for his constant interest and support throughout this project.  

\vspace{2mm}

\section{Branched affine structures on compact oriented surfaces.}

\subsection{Basics.}

A (branched) affine structure on $\Sigma$ is a $(\Aff, \C)$-structure with a finite number of singular points at which the structure is 'branched', which means that at such a singular point, the structure is the pull-back of ramified cover(at this point) of finite degree. For the sake of precision, we give two equivalent definitions of what an affine structure is:

\begin{defi}

A branched affine structure on $\Sigma$ is a atlas of charts $(U_j, \varphi_j)$ such that 

\begin{enumerate}
\item There exist a finite number of points $p_1, \cdots, p_n$, each belonging to only one $U_j$

\item Every time two open of charts $U_i$ and $U_j$ overlap, the transition map $\varphi_i \circ \varphi_j^{-1} : \varphi_j(U_i \cap U_j) \longrightarrow  \varphi_i(U_i \cap U_j)$ is the restriction of an element $\Aff$ to $\varphi_j(U_i \cap U_j)$ on each of the connected components of $\varphi_j(U_i \cap U_j)$.

\end{enumerate}

\end{defi}

Any chart can be analytically continued to define a 'super chart' on the universal cover of $\Sigma$ which is equivariant with respect to a representation of $\Gamma$ in $\Aff$. This remark leads to an alternative definition of a branched affine structure.

\begin{defi}

A branched affine structure on a compact Riemann surface $\Sigma$ is a non constant holomorphic function $\mathrm{dev} : \widetilde{\Sigma} \longrightarrow \mathbb{C}$ together with a group homomorphism $\rho : \Gamma \longrightarrow \mathrm{Aff}(\mathbb{C}) $ such that $\mathrm{dev}$ is $\rho$-equivariant, \textit{i.e.} satisfies that for every $z \in \widetilde{\Sigma}$ and every $\gamma \in \Gamma$, we have $$ \mathrm{dev}(\gamma \cdot z ) = \rho(\gamma)( \mathrm{dev}(z) ) $$

\noindent $\mathrm{dev}$ is called the developing map of the structure and $\rho$ the holonomy morphism of the structure.

\end{defi}

 Affine structures arise naturally either as generalisation of flat and translations surfaces (see \cite{Veech} for a very nice description of the structure of their moduli spaces, and also \cite{Veech2} for an investigation of their basic geometric properties), or as particular cases of branched projective structures whose holonomy has image an elementary subgroup of $\mathrm{PSL}(2,\C)$. A rather elementary but fundamental way to build affine structures is to glue along affine transformations the sides of a (collection of) polygon. 

\subsection{Holonomy}

Two affine surfaces are isomorphic if there exists a bijection between them which is affine when restricted to affine charts. In particular it implies that if $(\dev, \rho)$ is an affine structure on $\Sigma$ and $f \in \Aff$,  $f \circ \dev$ defines the same affine structure and the holonomy representation associated to $f \circ \dev$ is $f\circ \rho \circ f^{-1}$. Conversely two isomorphic affine structures on $\Sigma$ define holonomy representations which are only conjugated by an element of $\Aff$. The class up to conjugation of the holonomy representation therefore defines a geometric invariant of the structure. This invariant is known to be far from classifying; describing the set of surfaces having the same holonomy as been investigated in several contexts (see for instance \cite{BabaGupta} or \cite{CDF2} in the case of branched projective structures).

\vspace{2mm} The question we are going to investigate in this paper is to determine the representations that arise as the holonomy of a branched affine structure.

\subsection{Surgeries.}

We describe in this subsection procedures which we call surgery, which consist in  cutting  affine surfaces along several geodesic segments and gluing them back along those segments with different combinatorics in order to get new affine structures.

\subsubsection{Connected sum.}
\label{connected}
A geodesic line(resp. segment) on an affine surface is a path which is parametrised in any chart by a straight line(resp. segment). Since being a straight line(resp. segment) is invariant by $\Aff$, those objects are well-defined.

\noindent Consider two closed affine surfaces $\Sigma_1$ and $\Sigma_2$, as well as two geodesic segments $\gamma_1 \subset \Sigma_1$ and $\gamma_2 \subset \Sigma_2$. Cut along $\gamma_1$(resp. $\gamma2$) to get a surface $\Sigma_1'$(resp. $\Sigma_2'$) with a unique piecewise geodesic boundary component $\gamma_1^+ \cup \gamma_1^- $  (resp. $\gamma_2^+ \cup \gamma_2^- $). The surgery consists in gluing $\gamma_1^+$ to $\gamma_2^-$ and $\gamma_2^+$ to $\gamma_1^-$ respecting the affine structure. We get this way an affine structure on $\Sigma_1 \# \Sigma_2$ with two new branched points at the end of the image of $\gamma_1^+$ , both of angle $4\pi$.

\begin{figure}[!h]
\centering
\includegraphics[scale=0.8]{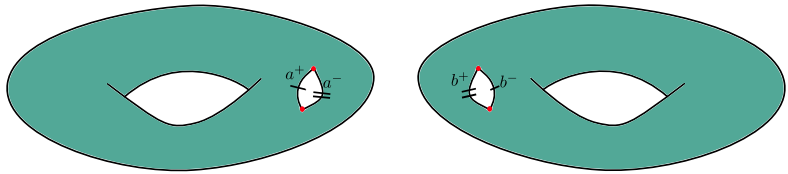}
\caption{The 'connected sum' surgery.}
\label{twotwists}
\end{figure}

\vspace{2mm}

A fact worth noticing is that the linear holonomy  $\alpha$ of this structure can be easily computed in terms of the ones of $\Sigma_1$ and $\Sigma_2$. The first homology group $\mathrm{H}_1(\Sigma_1 \# \Sigma_2, \Z)$ is equal to the product $\mathrm{H}_1(\Sigma_1, \Z) \times \mathrm{H}_1(\Sigma_2, \Z)$ and $\alpha : \mathrm{H}_1(\Sigma_1 \# \Sigma_2, \Z) \longrightarrow \C^*$ is equal to the product $\alpha_1 \times \alpha_2$. 

\subsubsection{Adding a handle.}

\label{adding}
There is a surgery, very similar to the previous one, that consists in creating a handle on an initial surface $\Sigma$ of genus $g$. Consider $a$ and $b$ two segments on $\Sigma$ and cut along these two segments. We get a new surface with two boundary components $a^+ \cup a^-$ and  $b^+ \cup b^-$.. Glue $a^+$ to $b^-$ and $a^-$ to $b^+$(along the unique affine transformations mapping $a^+$ to $b^-$ and $a^-$ to $b^+$)  to get a new affine structure on the compact surface of genus $g+1$. 

\vspace{2mm} Remark that on the new handle, a the loop has trivial holonomy (the one that surrounds  $a^+ \cup a^- = b^+ \cup b^-$).

\section{Character variety and mapping class group dynamics.}

\subsection{Character variety.}

We define in this section the character variety, which is roughly the set of representation of $\Gamma$ up to conjugation. We could directly take the quotient $\mathrm{Hom}(\Gamma, \Aff) / \Aff$ for the natural action of $\Aff$ by conjugation, which is the most natural thing to do. Unfortunately, this quotient is not very nice, it is not an analytic variety nor a smooth manifold and worst, is not even Hausdorff. This difficulty can be avoided by analyzing with little more care the structure of $\mathrm{Hom}(\Gamma, \Aff)$ and the action of $\Aff$.

\vspace{2mm}

Recall that $\Aff$ is canonically isomorphic to the semi direct product $\C^* \ltimes \C$ where $\C^*$ acts linearly on $\C$ as $\mathrm{GL}(1,\C)$. Therefore any representation $\rho : \Gamma \longrightarrow \Aff $ is the data of two functions $\Li_{\rho}$ and $\Tr_{\rho}$ such that 

\begin{enumerate}

\item $\Li_{\rho} : \Gamma \longrightarrow \C^*$ is a group homomorphism.

\item $\Tr_{\rho} : \Gamma \longrightarrow \C $ satisfies $\forall \gamma_1, \gamma_2 \in \Gamma$, $$ \Tr_{\rho}(\gamma_1 \cdot \gamma_2) = \Tr_{\rho}(\gamma_1) + \Li_{\rho}(\gamma_1)\Tr_{\rho}(\gamma_2)$$
\end{enumerate}

\begin{itemize}

\item  Since $\C^*$ is abelian, $\Li_{\rho}$ factorizes through 
$$ \Li_{\rho} : \mathrm{H}_1(\Sigma, \Z) \longrightarrow \C^* $$ because $\mathrm{H}_1(\Sigma, \Z)$ is the abelianisation of $\Gamma$. $\Li_{\rho}$ can then be seen as an element of $\Hl \simeq \mathrm{Hom}(\mathrm{H}_1(\Sigma, \Z), \C^*)$.

\item Let $\alpha$ be an element of $\Hl$. We define 

$$ \mathrm{Z}^1_{\alpha}(\Gamma, \mathbb{C}) = \{ \lambda : \Gamma \longrightarrow \mathbb{C} \ | \ \forall \gamma_1, \gamma_2 \in \Gamma, \ \lambda(\gamma_1\cdot \gamma_2) = \lambda(\gamma_1) + \alpha(\gamma_1)\lambda(\gamma_2)  \}  $$ 

$ \mathrm{Z}^1_{\alpha}(\Gamma, \mathbb{C})$ is a complex vector space. Its dimension can be computed in the following way : $\Gamma$ is a finitely generated group, with generators $a_1, b_1, \cdots, a_g, b_g$ and a unique relation between those generators $\prod_{i=1}^g{[a_i, b_i]} = 1$. Any element $\lambda \in \mathrm{Z}^1_{\alpha}(\Gamma, \mathbb{C})$ is characterised by its values $\lambda(a_1), \lambda(b_1), \cdots, \lambda(a_g), \lambda(b_g)$ and these numbers must satisfy the following relation (which is simply applying $\lambda$ to the relation on the generators) :

$$ \sum_{i=1}^g{  \lambda(a_i)(1 - \alpha(a_i) ) + \lambda(b_i)(1 - \alpha(b_i) )  }  = 0$$  

Conversely any data of $2g$ complex numbers satisfying the relation above defines an element of $ \mathrm{Z}^1_{\alpha}(\Gamma, \mathbb{C})$. The relation is trivial if and only if $\alpha \equiv 1$. Hence $\mathrm{Z}^1_{\alpha}(\Gamma, \mathbb{C})$ has complex dimension $2g$ if $\alpha \equiv 1$ and $2g -1$ in all other cases. 

\item For every $\rho \in \mathrm{Hom}(\Gamma, \Aff)$, $\Tr_{\rho}$ belongs to $\mathrm{Z}^1_{\Li_{\rho}}(\Gamma, \mathbb{C})$.

\end{itemize}

All the previous remarks can be rephrased in the following proposition : 

\vspace{3mm}

\begin{prop}

The projection $$ \pi : \mathrm{Hom}(\Gamma, \Aff) \longrightarrow \Hl $$ gives $\mathrm{Hom}(\Gamma, \Aff) \setminus \pi^{-1}(\{ 1 \})$ the structure of a complex vector bundle whose fiber has complex dimension $2g - 1$.  

\end{prop}

$\Aff$ acts by conjugation on $\mathrm{Hom}(\Gamma, \Aff)$: if $f \in \Aff$ and $\rho \in \mathrm{Hom}(\Gamma, \Aff)$, $(f \cdot \rho)(\gamma) = f \circ \rho(\gamma) \circ f^{-1}$. Notice that 

\begin{enumerate}

\item $\forall f \in \Aff$, $\pi(f\cdot \rho) = \pi(\rho)$.

\item $\forall \alpha \in \Hl$, $1- \alpha \in \mathrm{Z}^1_{\alpha}(\Gamma, \mathbb{C})$.

\item If $f = z \mapsto az + b$, $\Tr_{f \cdot \rho} = a \Tr_{\rho} + b(1 - \Li_{\rho})$.

\end{enumerate}

We introduce $\Ht =  \mathrm{Z}^1_{\alpha}(\Gamma, \mathbb{C}) / \C \cdot (1- \alpha) $. We use this specific notation because $\Ht$ is first cohomology group of $\Gamma$ twisted by $\alpha$. We also introduce  $\mathrm{Hom}'(\Gamma, \Aff)$ to be subset of  $\mathrm{Hom}(\Gamma, \Aff)$ made of representations whose image is not an abelian subgroup of $\Aff$. We define the character variety to be the equivalence classes of representations whose image is not an abelian subgroup under the action of $\Aff$ :

$$ \chi(\Gamma, \Aff) = \mathrm{Hom}'(\Gamma, \Aff) / \Aff $$

\noindent We then have the following description of the structure of $ \chi(\Gamma, \Aff)$:

\begin{prop}

The projection $$ \pi :\chi(\Gamma, \Aff) \longrightarrow \Hl $$ gives $\chi(\Gamma, \Aff)$ the structure of a $\mathbb{CP}^{2g-3}$-bundle over $\Hl \{ 1 \} \simeq (\C^*)^{2g} \setminus \{1, \cdots, 1 \}$.

\end{prop}

For a proof, see \cite{Ghazouani}, Proposition 1.2.

\subsection{Action of the mapping class group.}

The mapping class group is classicaly defined to be the group of connected components of $\mathrm{Diff}^+(\Sigma)$ the group of preserving orientation diffeomorphisms of $\Sigma$, namely 

$$ \Mod = \mathrm{Diff}^+(\Sigma) / \mathrm{Diff}_0(\Sigma)  = \mathrm{H}_0(\mathrm{Diff}^+(\Sigma)) $$

Recall that $\Gamma = \pi_1(\Sigma, *)$. A diffeomorphism $f \in \mathrm{Diff}^+(\Sigma)$ induces a group isomorphism :

$$ f_* : \pi_1(\Sigma, *) \longrightarrow \pi_1(\Sigma, f(*)) $$

\noindent An arbitrary choice of a path $c$ going from $*$ to $f(*)$ gives an identification between $\pi_1(\Sigma, *)$ and $\pi_1(\Sigma, f(*))$ and post-composing by such an identification gives us $f'_* : \Gamma \longrightarrow \Gamma$. Two such choices give automorphisms of $\Gamma$ which are conjugated, $f$ therefore defines an element of $\mathrm{Out}(\Gamma)$. Since $f$ is orientation preserving, any $f'_* : \Gamma \longrightarrow \Gamma$ preserves the fundamental class in $\mathrm{H}^2(\Gamma, \Z)$. We have then defined a group homomorphism : 

$$ \varphi : \Mod \longrightarrow \mathrm{Out}^+(\Gamma)$$ 

where $\mathrm{Out}^+(\Gamma)$ is the subgroup of $\mathrm{Out}(\Gamma)$ of elements preserving the fundamental class in $\mathrm{H}^2(\Gamma, \Z)$. A famous theorem of Dehn, Nielsen and Baer states that $\varphi$ is an isomorphism, see \cite[p.232]{FarbMargalit}. 

\vspace{4mm}

$\mathrm{Aut}(\Gamma)$ acts naturally on $\mathrm{Hom}(\Gamma, \Aff)$ by precomposition. We are going to prove in this paragraph that this action induces an action of $\mathrm{Out}(\Gamma)$ on $\car$.
\noindent Consider $\phi, \psi \in \mathrm{Aut}(\Gamma)$, note $[\phi], [\psi]$ their respective class in $\mathrm{Out}(\Gamma)$. Assume $[\phi] = [\psi]$. In other words, there exists $g \in \Gamma$ such that $\phi = g \cdot \psi \cdot g^{-1}$. Take $\rho \in \mathrm{Hom}'(\Gamma, \Aff)$. Then, $\forall \gamma \in \Gamma$ we have 

$$ \rho \circ \phi (\gamma) = \rho( g \psi(\gamma) g^{-1} ) = \rho(g) \cdot \rho\circ\psi(\gamma) \cdot \rho(g)^{-1} $$ 

Hence for all $\rho \in  \mathrm{Hom}'(\Gamma, \Aff)$, $\rho \circ \phi$ and $\rho \circ \psi$ belong to the same class in $\car$. This proves the $\mathrm{Aut}(\Gamma)$-action induces an action of $\mathrm{Out}(\Gamma) = \Mod$ on $\car$. We will refer to this action by \textit{the action of} $\Mod$ \textit{by precomposition}.

\subsection{The subset of holonomy of branched affine structures.}

We say of a representation $\rho : \Gamma \longrightarrow \Aff$ (or of its class in $\car$) that it is \textit{geometric} if it is the holonomy of a branched affine structure.

\begin{prop}
\label{openess}
The subset of geometric representations is an open subset of $\mathrm{Hom}(\Gamma, \mathrm{Aff}(\mathbb{C}))$. Its projection is therefore open in $\car$ and it is $\Mod$-invariant.

\end{prop}

\begin{proof}

The $\Mod$ invariance is rahter easy. Consider an affine structure of (class of) holonomy $\rho$ and  $f$ a diffeomorphism of $\Sigma$. The affine structure pulled-back by $f$ has holonomy $[\rho \circ f^*]$,  and $[\rho \circ f^*]$ is therefore geometric. 

\vspace{2mm}

We now explain why the general Ehresmann-Weil-Thurston principle for geometric structures ensures that the set of geometric representations is open. Let $\rho_0$ be the holonomy representation of an affine structure on $\Sigma$ and $U \subset \mathrm{Hom}(\Gamma, \mathrm{Aff}(\mathbb{C}))$ an open subset containing $\rho_0$. The group $\Gamma$ acts properly and discontinuously on  $U \times \tilde{\Sigma} \times \mathbb{CP}^1 $ the following way

$$\begin{array}{ccc}

U \times \tilde{\Sigma} \times \mathbb{CP}^1  & \longrightarrow  & U \times \tilde{\Sigma} \times \mathbb{CP}^1  \\

(\rho, x, z)                       & \longmapsto      &  (\rho, \gamma \cdot x, \rho (\gamma) \cdot z)
\end{array}$$

\noindent We denote the quotient of this action $E$. The natural projection $E \rightarrow U$ is a submersion whose fibers are compact (they are all diffeomorphic to $\Sigma \times \mathbb{CP}^1$). Ehresmann's theorem ensures that it is actually a fiber bundle whose fiber over $\rho$ we denote $M_{\rho}$. It is itself a flat bundle in $\mathbb{CP}^1$ over $\Sigma$, whose monodromy is exactly $\rho$. We denote by $\mathcal{F}_{\rho}$ the foliation associated to its flat connection.

\vspace{2mm}

The projection on the factor $\Sigma$ is also a fiber bundle of fiber $U \times \mathbb{C}P^1$. Eheresmann's theorem together with the continuous family of foliations $\mathcal{F}_{\rho}$ provide for each point $p \in S$ an open set $\Omega_p \ni p$ such that a neighbourhood of $(\rho_0, p ) \times \mathbb{CP}^1 $ in $E$ has product structure of the form  $U' \times \Omega  \times \mathbb{CP}^1$ such that the sets ${\rho} \times \Omega \times \{ z \}$ are included in leaves of $\mathcal{F}_{\rho}$.

We are now set to prove that an affine structure of holonomy $\rho_0$ can be deformed to an affine structure of holonomy $\rho$ for $\rho$ close to $\rho_0$. Recall that a branched affine structure on $\Sigma$ of holonomy $\rho$ is a section of $M_{\rho}$ which is transverse to $\mathcal{F}_{\rho}$ except at a finite number of points where it is tangent to the foliation at a finite order. Consider an arbitrary trivialisation $V \times M_{\rho_0}$ of $E$ above a neighbourhood $V$ of $\rho_0$. The graph of a section $s_0$ of $M_{\rho_0}$ can be pushed to each $M_{\rho}$ for $\rho \in V$ by means of this trivialisation. Using the trivialisations defined in the next paragraph, it is easy to see that if $s_0$ was transverse to $\mathcal{F}_{\rho}$ except at a finite number of points, so are its pushed forwards. Moreover, at the finite number of points, the pushed forward must have same order of tangency and they therefore define affine structures of holonomy $\rho$ for all $\rho \in V$. 

\end{proof}

\subsection{Euclidean representations.}

\label{euclidean}

\subsubsection{The volume of a Euclidean representation.}
Euclidean representations are in some way essentially different from stricly affine ones. A remarkable fact is that we can, for such Euclidean representations, define an invariant called the \textbf{volume}. For a Euclidean representation $\rho$, it is the total volume of the pull back of volume form of $\C$ by any $\rho$-equivariant map $\widetilde{\Sigma} \longrightarrow \C$, see \cite{Ghazouani} for further details on this construction. \\ There is a more cohomological way to define this \textbf{volume}. A Euclidean representation can be thought of as an element of  $\mathrm{Z}^1_{\alpha}(\Gamma, \mathbb{C})$. The cup product 

$$ \wedge : \Ht \times \mathrm{H}_{\overline{\alpha}}(\Gamma, \C) \longrightarrow \mathrm{H}^2(\Gamma, \C) = \C $$ 

\noindent is a non-degenerated bilinear form (because of the Poincar\'e duality) which makes

$$\begin{array}{ccccc}
\mathrm{vol} & : & \Ht \times \Ht & \longrightarrow &  \mathrm{H}^2(\Gamma, \C) = \C \\
      &    &    (\mu, \nu) & \longmapsto &  \mu \wedge \overline{\nu} 
\end{array}  $$ 

\noindent a non-degenerated Hermitian form of signature $(g-1,g-1)$ (see \cite{Ghazouani}, Proposition 6.2). Since a point a Euclidean representation in $\chi$ is a point in $\mathbf{P}(\Ht)$, it makes sense to say that is has \textit{positive}, \textit{null} or \textit{negative} volume. This \textit{'sign'} is $\Mod$-invariant; in particular a Euclidean geometric representation must have positive volume. 

\subsubsection{The Torelli group action}

The natural action of the Torelli group on $\Ht$ must preserve the form $\mathrm{vol}$. This therefore defines for every unitary $\alpha$ a representation $r_{\alpha}$

$$ r_{\alpha} : \mathcal{I}(\Sigma) \longrightarrow \mathrm{PU}(\mathrm{vol}) \simeq  \mathrm{PU}(g-1, g-1).$$ 

We are going to come back to this action in careful detail in Section \ref{torelli}

\section{The $\mathrm{Mod}(\Sigma)$-action on $\mathrm{H}^1(\Sigma, \mathbb{C}^*)$.}

\label{modlinear}

The $\Mod$-action on $\Hl$ is the action by precomposition, when thinking of $\Hl$ as $\mathrm{Hom}(\mathrm{H}_1(\Sigma, \Z), \mathbb{C}^*)$. Up to the choice of a symplectic basis of  $\mathrm{H}_1(\Sigma, \Z)$, $\mathrm{H}_1(\Sigma, \C^*)$ identifies to $(\mathbb{C}^*)^{2g} \simeq \mathbb{R}^{2g} \times \mathbb{R}^{2g}/ \Z^{2g}$ and the $\Mod$-action factors through the diagonal linear action of $\Sp$ on $\mathbb{R}^{2g} \times \mathbb{R}^{2g}/ \Z^{2g}$. 
We describe in this Section the possible orbit closures of this action.

\subsection{Closed subgroups of $\C^*$ and closed invariant subsets of $\Hl$.}

Subsets of $\Hl$ which consists of elements $\alpha$ such that $\mathrm{Im}(\alpha) \subset H$, where $H$ is a closed subgroup of $\C^*$ are closed in $\Hl$ and invariant under the action of the mapping class group. A first step in the understanding of orbit closures of the $\Mod$-action on $\Hl$ is therefore to list such subgroups $H$.

\begin{prop}
Let $H$ be a closed subgroup of $\C^*$.

\begin{enumerate}

\item Either $H$ is a finite subgroup spanned by a root of unity; \sk

\item or $H$ is discrete and of the form $ \{ e^{na} \ | \ n \in \Z\}$ for some complex number $a$; \sk

\item or $H$ is a $1$-parameter subgroup of the from $H_a = \{ e^{ta} \ | \ t \in \R\}$; \sk

\item or $H$ is the product of one of the last two type by a a finite subgroup spanned by a root of unity; \sk

\item or $H$ is $\C^*$.

\end{enumerate}

\end{prop}

\begin{proof}

$\mathbb{C}^*$ being a two dimensional Lie group, the proposition follows from the fact that every closed subgroup of a Lie group is a Lie subgroup.

\end{proof}

Each subgroup of this list gives rise to an invariant closed subset by the $\Mod$-action. Except for the class of non unitary characters whose modulus has discrete image in $\R_+^*$, we are going to prove that these are the only closed invariant subset of $\Hl$. 

\vspace{2mm}

\paragraph{\bf Other invariant subsets.} In the case where $\mathrm{Im(|\alpha|})$ is discrete and non-trivial, there exists $m>0$ such that the character $\mu \in \mathrm{H}^1(\Sigma,\R)$, which is the lift of $|\alpha|$, has values in $m\Z$. Let $\tilde{\theta} \in \mathrm{H}^1(\Sigma,\R)$ be a lift of $\theta = \mathrm{arg}(\alpha)$. Such a lift is unique up to an element of $\mathrm{H}^1(\Sigma, \Z)$. The value $v = \mu \wedge \tilde{\theta}$ modulo $m\Z$(which we denote by $\mu \wedge \theta$) is therefore invariant of the $\Mod$-action. The level sets of the function $\mu \wedge \cdot$ being closed, they therefore define closed invariant subset for the action of the mapping class group. Such a set is denoted by $\mathrm{H}_{m,v}$.
\noindent We are going to prove that the closed invariant subset that we have just described are the only one. More precisely:

\begin{prop}

\begin{itemize}

\item If $\mathrm{Im}(|\alpha|)$ is discrete and non trivial and is equal to $\langle e^{m} \rangle$ with $m>0$, and that $\mu \wedge \mathrm{arg}(\alpha) = v$, then $\overline{\Mod \cdot \alpha} = \mathrm{H}_{m,v}$, where $\mathrm{H}_{m,v}$ is the set of characters whose modulus has image $\langle e^{m} \rangle$, and for which the cup product of the lift of their modulus and their argument is $v \in \mathbb{R}/m\Z$.

\item If $\alpha$ is unitary or if  $\mathrm{Im}(|\alpha|)$ is not discrete, then $\overline{\Mod \cdot \alpha}$ is the set of characters $\beta$ such that $\mathrm{Im}(\beta) \subset \overline{\mathrm{Im}(\alpha)}$.

\end{itemize}

\end{prop}

The remainder of the Section will be devoted to its proof.

\subsection{Ratner's theorem.}

We are interested in classifying the possible orbit closures of an element of $\mathbb{R}^{2g} \times \mathbb{R}^{2g}/ \Z^{2g}$ under the $\Sp$-action. To that purpose, we remark that a subset $A$ of $\mathbb{R}^{2g} \times \mathbb{R}^{2g}$ is invariant under the action of $\Sp \ltimes \Z^{2g}$ if and only if it projection $p(A)$ on $\mathbb{R}^{2g} \times \mathbb{R}^{2g}/ \Z^{2g}$ is invariant under the $\Sp$-action. 

\vspace{2mm}

We can also remark that since $G = \Sr \ltimes \mathbb{R}^{2g}$ acts transitively on $X = \mathbb{R}^{2g} \times \mathbb{R}^{2g}$, there is a natural identification between $X$ and $G/U$ where $U$ is the stabilizer in $G$ of a point in $X$. Since $\Gamma = \Sp \ltimes \Z^{2g}$ is a lattice in $G$, we are going to be able to apply the Ratner's powerful theorem to our setting.  Rephrased in our specific context, it can be stated as follow:

\begin{thm}[Ratner, \cite{Ratner}]
\label{ratner}
Let $G, U$ and $\Gamma$ be as above and $p \in X = G/U$ such that $p =gU$. Then there exists a closed subgroup $H_g < G $ such that 

\begin{itemize}

\item $U^g = gUg^{-1} \subset H_g$; \sk

\item $\Gamma \cap H_g$ is a lattice in $H_g$; \sk

\item $\overline{\Gamma \cdot p} = \Gamma H p$.

\end{itemize}

\end{thm}

The first step towards the orbit classification of the $\Mod$-action on $\Hl$ is to list those subgroups $H < \Sr \ltimes \R^{2g}$ which contains $U = U' \times \{0\}$ where $U' < \Sr$ is the stabiliser of the point $\vec{e} = \tiny{ \begin{pmatrix} 
 1 \\ 
 0 \\
 \vdots \\
 0  
 \end{pmatrix}}$.

\vspace{3mm}

We fix $p =(\vec{x}, \vec{u})\in \R^{2g}\times \R^{2g}/ \Z^{2g} $ and $g$ being such that $p =gU$ or in other word $(\vec{x}, \vec{u}) = g \cdot (\vec{e}, \vec{0})$. Let $H_g$ be the group given by Ratner's theorem. Denote by $A_g< \Sr$ the image of $H_g$ under the projection $ \Sr \ltimes \R^{2g} \longrightarrow \Sr$ and $B_g < \R^{2g}$ the kernel of this projection. We first prove two lemmas that will help us traduce the two first conditions of Theorem \ref{ratner} in terms of $A_g$ and $B_g$.

\begin{lemma}

$B_g$ is either $\vec 0$, $\R\vec x$, $\vec x^{\perp}$ or $\R^{2g}$.

\end{lemma}

\begin{proof} Remark that $B_g$ is invariant under the action of $U_g$ by conjugation (which is nothing but the natural linear action of $U_g$). Let $V$ be an invariant subvector space of $\R^{2g}$, four possibilities can occur:

\begin{enumerate}

\item $V = \vec{0}$; \sk

\item $V = \vec{x}$;

\item $V$ contains $\vec y$ such that $\vec{x} \wedge \vec y \neq 0$. $U_g$ acts transitively on the set of lines on which $\vec{x} \wedge \cdot $ is non zero, which is dense in $\R^{2g}$ and therefore  $V = \R^{2g}$.

\item $V \subset \vec x^{\perp}$ and  $V \neq \vec{x}$. Since $U_g$ acts transitively on vectors orthogonal to $\vec{x}$  which are not on the line directed by $\vec{x}$, $V =  \vec x^{\perp}$.

\end{enumerate}

\end{proof}
\begin{lemma}
\label{volume}
Let $H $ be a closed Lie subgroup of $G = \Sr \ltimes \R^{2g}$. Then $\Sp \ltimes \Z^{2g} \cap H$ is a lattice in $H$ if and only if $\Sp\cap A$ is a lattice in $A$ and $\Z^{2g} \cap B$ is a lattice in $B$.

\end{lemma}

\begin{proof} Denote by $\pi$ the natural projection:

$$  H/ (H \cap \Gamma) \longrightarrow A / \Sp\cap A $$

\noindent This projection is a fiber bundle whose fiber is isomorphic to $B / (\Z^{2g} \cap B)$. The Haar measure on $H$ is the product of the Lebesgue measure on the fiber and the Haar measure $A / (A \cap \Sp)$. Therefore the volume of $H$ is exactly $\mathrm{vol}(A / (A \cap \Sp)) \cdot \mathrm{vol}(B / (\Z^{2g} \cap B))$.

\end{proof}

According to the previous lemma, if $H_g$ associated to the orbit closure of $p \in X$,  $A_g$ is unimodular and contains $\pi(U_g)$ which is $\mathrm{Stab}(\vec{x})$. According to \cite{Kapovich} (or \cite{Deroin} for a clearer exposition of the arguments),  $H_g$ is either $U_g$ or $\Sr$. We deduce from the two lemmas above the following proposition.

\begin{prop}

$$ H_g =  A_g \ltimes B_g $$ 

\end{prop}

\begin{proof}

Without loss of generality we can assume that $p = (\vec{e}, \vec{0}) $ and $U = \mathrm{Stab}(p)$. If $A_g = U$ then it is obvious (because $H_g$ contains $U \ltimes \{\vec{0} \}$). We assume from now on that $A_g = \Sr$. We are first going to do the proof in the case when $B = \vec 0$. In this case the projection $ \Sr \ltimes \R^{2g} \longrightarrow \Sr$ is one-to-one and its inverse defines by projecting on the $\R^{2g}$ factor a continuous function $ \varphi : \Sr \longrightarrow \R^{2g} $ such that $$ \forall A, B \in \Sr,  \ \varphi(AB) = \varphi(A) + A \varphi(B) $$ \noindent Proving the proposition in this specific case is equivalent to prove that $\varphi$ vanishes everywhere.  Since $\varphi$ vanishes on $U$, it defines $ \varphi : \Sr / U \simeq \R^{2g} \setminus 0 \longrightarrow \R^{2g} $ such that 

$$ \forall A \in \Sr \ \text{and} \ \forall  y \in \R^{2g} \setminus 0, \ \varphi(Ay) = \varphi(A_1) + A \varphi(y) $$ \noindent and $\varphi(\vec{e}) =0$, where $A_1$ denotes the first column of the matrix $A$. Considering $\vec z$ such that $\vec z \wedge p \neq 0$. One can find a element $A \in U$ such that $\mathrm{Ker}(A - \mathrm{Id}) = \mathrm{Span}(\vec y, \vec{e})$. Since $\varphi(A) = 0$, $\varphi(\vec y)$ must be a fixed point of $A$ and therefore lies in $ \mathrm{Span}(\vec y, \vec{e})$. Since the set of $\vec y$ such that $\vec z \wedge \vec y \neq 0$ is dense, every $\varphi(\vec y)$ lies in $\mathrm{Span}(\vec y, \vec{e})$.

\noindent We want to show that $\varphi(\vec y) = 0$ for all $\vec y$. Consider $A \in \Sr$ such that $A\vec{e} = \vec y$, \textit{i.e.} $A_1 = \vec y$. Consider $\vec z$ such that $A\vec z$ does not belong to $ \mathrm{Span}(\vec e, \vec{y})$. On one hand 

$$ \varphi(A\vec z) = \varphi(\vec{y}) + A \varphi(\vec z) $$

\noindent but also

$$ \varphi(A\vec z) = \lambda A \vec z + \mu \vec{e} .$$

This implies that $\varphi(\vec{y})$ is a multiple of $\vec{e}$, which holds for all $\vec y$. But then $\varphi(A \vec y) = \varphi(A_1) + A \varphi(\vec y)$ and therefore $\varphi(\vec y)$ must be zero because $\varphi(A \vec y)$ is also a multiple of $\vec{e}$. 
\\ When $B_g \neq 0$, one can still play the same game. Notice that the set of $\vec{z}$ such that $(A,\vec{z})$ belongs to $H_g$ is an affine subspace of tangent subspace $B_g$ and we can still define $ \varphi : \Sr \longrightarrow \R^{2g}/B_g$ which vanishes on $U_g$. A similar argument left to the reader gives that $\varphi$ vanishes everywhere and therefore proves the proposition.
\end{proof}

We deduce then from the lemma above that a subgroup of $G$ which contains $U_g \ltimes \{ \vec{0}\}= \Stab(p)$ with $p=(\vec{x}, \vec{u})$ is either $U_g \ltimes \{0\}$, $U_g \ltimes \R \vec{x}$, $U_g \ltimes \vec{x}^{\perp}$, $\Sr \ltimes \{0\} $ or $\Sr \ltimes \R^{2g}$. For the sake of the exposition, we will denote in what follows by $H_0 = A_0 \ltimes B_0$ an arbitrary subgroup of the list for $\vec{x} = \vec{e}$.

\subsection{Orbit closures classification.}
\label{orbitclosure}
We are now set to classify the orbit closures of the action of $\Gamma$ on $\Hl = \mathbb{R}^{2g} \times \mathbb{R}^{2g}/ \Z^{2g}$. We consider $\tilde{p} = (\vec{x}, \vec{u}) \in \R^{2g} \times \R^{2g}$ representing a point $p\in \Hl$, and we make the extra assumption that $\vec x \neq 0$. In this case, we can apply Ratner's theorem to $\tilde{p} =gU$ where $ g: \vec z \mapsto C\vec z + \vec u $ and $C$ is such that its first column is $\vec{x}$. For one of the subgroups $H_0 = A_0 \ltimes B_0$ listed above, we have that $\overline{\Gamma\tilde{p}} = \Gamma gH_0g^{-1} \tilde{p}$ and that $\Gamma \cap  gH_0g^{-1}$ is a lattice in $gH_0g^{-1} = H_g$.
\noindent The conjugation map $H_0 \longrightarrow H_g$ writes down explicitly 
$$ (M, \vec y) \longmapsto (CMC^{-1}, \vec u - CMC^{-1} \vec u + C \vec y). $$ 

\noindent Therefore $U_g = gUg^{-1} = \mathrm{Stab}(\vec{x})$ and $B_g = C\cdot B_0$. We now classify orbits depending on whether $A_g = \mathrm{Stab}(\vec{x})$ or $\Sr$.
\vspace{3mm}
\paragraph{  \it 1) If $A_g= \mathrm{Stab}(\vec x)$.}

\vspace{2mm} Since $\Sp$ must be a lattice in $A_g$ (see Lemma \ref{volume}), $\vec{x}$ must belong to a rational line, \textit{i.e.} there exists $\vec{x'} \in \mathbb{Z}^n$ such that $\vec x = \lambda \vec{x'}$ for a constant $\lambda > 0$. This is equivalent to say that the absolute value of elements of $\mathrm{Im}(p)$ is the (discrete) group spanned by $e^{\lambda}$ if $\vec x'$ has been chosen primitive. In this case four possibilities can occur:

\begin{itemize}

\item $B_g = 0 $. In that case $\overline{\gamma \tilde{p}} = \gamma \tilde{p} $  which is equivalent to the fact that $\mathrm{Im}(p)$ is a discrete subgroup of $\C^*$. 

\item $B_g = \R \vec x$. Since $A_g = \mathrm{Stab}(\vec x)$, the image in $\R^{2g} \times \R^{2g}$ of $H_g \cdot p$ is equal to $\{ \vec{x}  \} \times \{ \Stab(\vec{x})\vec{u} + \R \vec{x} \} = \{ \vec{x}  \} \times \{ \vec{u} + \vec{x}^{\perp} \}$.  If $\vec{u}$ does not belong to $\R \vec{x}$, then $H_g \cdot p = \{ \vec{x}  \} \times \{ \vec{y} \ | \ \vec{y} \wedge \vec{x} =   \vec{u} \wedge \vec{x} \} $ and thus $\Gamma\cdot p = H_{\lambda,v}$ where $v= \vec{u} \wedge \vec{x} \ \mathrm{mod} \lambda $.
\noindent If $\vec{u} \in \R \vec{x}$ it implies that $\mathrm{Im}(p)$ is a discrete subgroup of the form $\{  e^{n\lambda + n2i\pi \mu}  \ | \ n \in \Z  \}$ where $\mu$ is irrational.

\item $B_g = \vec{x}^{\perp}$ works the exact same way than the previous case, one remarks that the projection $H_g \longrightarrow \Hl$ maps to $\{ \vec{x} \} \times \{ \vec{u} + \vec{x}^{\perp} \}$ and therefore the same conclusion holds.

\item The last remaining subcase $B_g = \R^{2g}$ cannot occur since we already proved that the closure of the orbit of $p$ must be contained in an invariant subset of the type $H_{m,v}$.
\end{itemize}

So far we have classified all possible orbit closures of $p$ in the case when the image its modulus is a given discrete subgroup of $\mathbb{R}^*$. Three situations can occur:

\begin{enumerate}

\item$p$ has discrete image and its image is isomorphic to $\Z \times \Z/n\Z$ through $(k,k') \mapsto e^{kz_0 + \frac{2i\pi k'}{n}} $ for a certain $z_0 \in \mathbb{C}^*$; \sk

\item its rotational part  has not discrete image in $\mathbb{U}$ and in this case its orbit closure is the set $\mathrm{H_{\lambda,v}}$; \sk

\end{enumerate}

\paragraph{ \it 2) If $A_g= \Sr$.} In this case, $B_0$ is either $\{ 0 \}$ or $\R^{2g}$. One easily deduce from the fact that $A_g = \Sr$ that the modulus of $p$ has dense image in $\R^*_+$.  Using similar arguments to case $1)$, we find that:

\begin{enumerate}

\item either the rotational part of $p$ has finite image $\langle e^{\frac{2i\pi}{n}} \rangle $, in this case $\overline{\mathrm{Im}(p)} = \R^*_+ \times \langle e^{\frac{2i\pi}{n}} \rangle$ and the closure of the orbit of $p$ is the set of elements whose closure of the image is $L\times \langle e^{\frac{2i\pi}{n}} \rangle$, where $L$ is a one parameter subgroup of $\C^*$, of the form $\{ e^{tz_0} \ | \ t \in \R \}$ for a certain $z_0 \neq 0$ ;\sk

\item or the rotational part of $p$ has dense image in $\U$, in this case $\overline{\mathrm{Im}(p)}  = \C^*$ and the closure of the orbit of $p$ is the set of elements whose closure of the image is $\C^*$.

\end{enumerate}

The latter case is actually the generic case relatively to the Lebesgue measure.

\vspace{4mm}

\paragraph{ \it 3) If the modulus of $p$ is trivial, \textit{i.e.} $p$ is Euclidean.}

This case is the easiest. We let the reader verify that a direct application of Ratner's theorem to the case $\mathbb{R}^{2g} \setminus \{0\} = \Sr / \Stab(\vec{e})$ for the $\Sp$-action gives the following dichotomy:

\begin{enumerate}

\item either $\mathrm{Im}(p)$ is a finite subgroup of $\U$ and in this case the orbit of $p$ is discrete and consist of the representations having the same image.

\item or $\mathrm{Im}(p)$ is dense in $\mathbb{U}$ and in this case the closure of the orbit of $p$ is the set of Euclidean elements of $\Hl$.

\end{enumerate}

\section{The Torelli group action on $\mathbf{P}\mathrm{H}^1_{\alpha}(\Gamma, \mathbb{C})$.}
\label{torelli}

The kernel of the symplectic representation of $\mathrm{Mod}(\Sigma)$ is called the \textit{Torelli group} and is denoted by $\mathcal{I}(\Sigma)$. A remarkable fact is that the action of $\mathcal{I}(\Sigma)$ stabilises globally the fibers of the bundle 

$$ \car \longrightarrow \Hl.$$

For every parameter $\alpha \in \Hl$, $\Stab(\alpha)$ contains the Torelli group and the $\Stab(\alpha)$-action induces a projective linear representation

$$ \tau_{\alpha} : \Stab(\alpha) \longrightarrow \mathrm{Aut}(\mathcal{L}_{\alpha}) \simeq \mathrm{PGL}(2g-2, \mathbb{C}).$$ 

\noindent called the \textit{Chueshev representation}. It was first introduced by Chueshev in \cite{Chu}. 

In \cite{Ghazouani}, we computed explicitly the action of a family of Dehn twists along $2g - 2$ separating curves, which allowed us to prove the ergodicity of the $\Mod$ action. We analyse with more careful detail this action in order to describe the closure of its image in $\mathrm{PGL}(2g-2, \mathbb{C}) $. We prove the following theorem whose proof the whole Section is dedicated to.

\begin{thm}

\label{image}
\begin{enumerate}

\item If $\mathrm{Im}(\alpha)$ is a non-trivial subgroup of $\R^*_+$, then $r_{\alpha}(\mathrm{Stab}(\alpha))$  is dense in $\mathrm{PGL}(2g-2, \mathbb{R}) \subset  \mathrm{PGL}(2g-2, \mathbb{C}) \simeq \mathrm{PGL}(\Ht) $.

\item If $\mathrm{Im}(\alpha) \subset \mathbb{U} $ and is different from $\langle \exp(\frac{2i\pi}{n}) \rangle$ for $n= 1,2,3,4,6$, then $r_{\alpha}(\mathrm{Stab}(\alpha))$  is dense in $\mathrm{PU}(g-1, g-1) \subset  \mathrm{PGL}(2g-2, \mathbb{C}) \simeq \mathrm{PGL}(\Ht) $.

\item  If $ \mathrm{Im}(\alpha)$ is any other subgroup of $\C^*$, then $r_{\alpha}(\mathrm{Stab}(\alpha))$ is dense in $\mathrm{PGL}(2g-2, \mathbb{C})\subset \mathrm{PGL}(\Ht)$.

\end{enumerate}

\end{thm}

\subsection{Remarkable elements of $\Ht$}
\label{remarkable}
Let $\delta \subset \Sigma$ be a separating simple curve. For any $\alpha \in \Hl$, we define $\mu_{\delta}$ an element of $\Ht$ associated to $\delta$. Let $p$ be the base point of $\pi_1 \Sigma = \Gamma$. We make the assumption that $p \in \delta$. Any class $[\gamma] \in \Gamma$ can be represented by a closed curve $\gamma$ based at $p$ such that $\gamma$ intersects $\delta$ transversally and a finite number of time. \\ The curve $\delta$ separates $\Sigma$ into two components $\Sigma_+$ and $\Sigma_-$. It can be decomposed in a finite number of curves $\gamma_1, \cdots,\gamma_k$ whose end points lie on $\delta$ such that for all $i\leq k$, $\gamma_i$ is entirely contained in $\Sigma_+$ or $\Sigma_-$. Let $\beta_i$ a closed curve that one get joining the end points of $\gamma_i$ by a portion of $\delta$. Remark that the homology class of $\beta_i$ does not depend on the way we close $\gamma_i$ because $\delta$ is separating and hence homologically trivial. We set by definition 

$$ \mu_{\delta}(\gamma) = (-1)^{\epsilon} + \sum_{i=1}^k{(-1)^{\epsilon+i}\prod_{j\leq i}{\alpha(\beta_j)}} $$

\noindent This formula defines an element $\mu_{\delta}$ of  $\mathrm{Z}^1_{\alpha}(\Gamma, \mathbb{C})$ which is trivial in $\Ht$ if and only if the restriction of $\alpha$ to either $\Sigma_+$ or $\Sigma_-$ is trivial. This construction of $\mu_{\delta}$ can be extended to the case where $\delta$ is such that $\alpha(\delta)=1$, but the construction is slightly more involved. It is done in detail in \cite[Lemma 3.2]{Ghazouani}. We will use this fact in a crucial way at the end of the section, see Lemma \ref{particular}.

\vspace{4mm}

We give now a more conceptual construction of $\mu_{\delta}$. $\delta$ is a separating curve, the two surfaces $\Sigma_+$ and $\Sigma_-$ are therefore such that $\Sigma_+ \cap \Sigma_- = \delta$. The Mayer-Vietoris exact sequence for the twisted cohomology of $\Sigma$ associated to the partition $\Sigma = \Sigma_+ \cup \Sigma_-$ is the following

$$ 0 \rightarrow  H^0_{\alpha}(\Sigma) \rightarrow H^0_{\alpha}(\Sigma_+) \oplus H^0_{\alpha}(\Sigma_-) \rightarrow H^0_{\alpha}(\delta) \rightarrow H^1_{\alpha}(\Sigma)  \rightarrow  H^1_{\alpha}(\Sigma_+) \oplus H^1_{\alpha}(\Sigma_-) \rightarrow H^1_{\alpha}(\delta) \rightarrow H^2_{\alpha}(\Sigma)  \rightarrow \cdots$$

If $\alpha$ is non-trivial restricted to both $\Sigma_+$ and $\Sigma_-$, the two first factors of the sequence vanish. We are left with an injective morphism

$$ H^0_{\alpha}(\delta) \rightarrow H^1_{\alpha}(\Sigma) $$ whose image of the fundamental class of $ H^0_{\alpha}(\delta)$ in $ H^1_{\alpha}(\Sigma) $ is $\mu_{\delta}$.

\subsection{Fomula for the action of a Dehn twist.}

\begin{prop}
\label{action}
Let $\delta \subset \Sigma$ be a simple curve such that $\alpha(\delta) = 1$. For $\mu_{\delta} $ as above, the action of $T_{\delta}$ on $\Ht$ is a transvection of vector $\mu_{\delta}$. More precisely, for any $\lambda \in \mathrm{Z}^1_{\alpha}(\Gamma, \mathbb{C})$,

$$ T_{\delta}\cdot \lambda = \lambda + \lambda(\delta)\mu_{\delta} $$

\end{prop}

Once again, the proof can be found in \cite{Ghazouani}, Lemma 3.2.

\subsection{Genus two representations.}

The genus $2$ is the model to be understood, for it contains all the geometrical difficulty. We therefore make the assumption that $\Sigma$ has genus $2$ for the remainder of this subsection. Depending on whether $\alpha$ is totally real, unitary or generic the image of the Torelli group falls in three distinct subgroups of $\mathrm{PSL}(2, \C)$, respectively $\mathrm{PSL}(2, \R)$, $\mathrm{PU}(1,1)$ or $\mathrm{PSL}(2, \C)$. We aim at describing more precisely the image depending on $\alpha$. We are going to prove that except for a finite number of exceptions, the image of $\Stab(\alpha)$ is dense in the subgroup of $\mathrm{PSL}(2, \C)$ in which it is obviously contained.

\begin{lemma}
\label{Torelli2}
\begin{enumerate}

\item If $\alpha \in \Hl$ has image a non-trivial subgroup of $\mathbb{R}^*$., then the image of $\Stab(\alpha)$ under the Chueshev representation has dense image in $\mathrm{PSL}(2,\R) \subset  \mathrm{PSL}(2,\C)  \simeq \mathrm{PGL}(\Ht)$.

\item If $\alpha \in \Hl$ has image a non-trivial subgroup of $\mathbb{C}^*$ not contained in $\U$, then the image of $\Stab(\alpha)$ under the Chueshev representation has dense image in $\mathrm{PSL}(2,\C) \simeq \mathrm{PGL}(\Ht)$.

\item If $\alpha \in \Hl$ has image a subgroup of $\U$ which is different from $\langle \exp(\frac{2i\pi}{n}) \rangle$ for $n= 1,2,3,4,6$, then the image of $\Stab(\alpha)$ under the Chueshev representation has dense image in $\mathrm{PU}(1,1) \subset \mathrm{PSL}(2,\C)  \simeq  \mathrm{PGL}(\Ht)$.

\item If $\alpha \in \Hl$ has image $\langle \exp(\frac{2i\pi}{n}) \rangle$ for $n= 1,2,3,4,6$, then the image of $\Stab(\alpha)$ under the Chueshev representation a discrete subgroup of $\mathrm{PU}(1,1) \subset \mathrm{PSL}(2,\C)  \simeq   \mathrm{PGL}(\Ht)$.
\end{enumerate}
\end{lemma}

We dedicate the rest of the subsection to the proof of Lemma \ref{Torelli2}, which will require examination of several cases. Most of them will be dealt with the calculation of the action of two particular Dehn twists along separating curves, as it had previously been done in \cite{Ghazouani}.

\paragraph{\bf Action of two Dehn twists along separating curves.}
\label{2dehntwists}

If the curves $a$ and $b$ are such that both $A=\alpha(a)$ and $B=\alpha(b)$ are different from $1$, we have proven in \cite[Section 3.2]{Ghazouani} that the matrices of the two Dehn twists around the cuves of Figure \ref{twotwists} are
$$ \begin{pmatrix}

1 & (1-A)(1-B) \\
0 & 1
\end{pmatrix} \begin{pmatrix}

1 & 0 \\
(1-A^{-1})(1-B^{-1}) & 1
\end{pmatrix} $$
\begin{figure}[!h]
\centering
\includegraphics[scale=0.7]{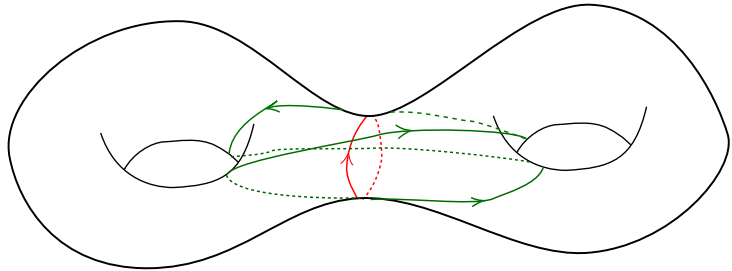}
\caption{Two separating curves.}
\label{twotwists}
\end{figure}

When the product $(2 - A - A^{-1})(2 - B - B^{-1})$ is strictly smaller than $1$, these two matrices generate a  non-discrete, non-elementary subgroup of $\mathrm{PSL}(2,\mathbb{C})$, according to Jorgensen's lemma (see \cite{Jorgensen}). Its closure is a non-elementary Lie subgroup of $\mathrm{PSL}(2,\mathbb{C})$ and is either $\mathrm{PSL}(2,\mathbb{C})$ or $\mathrm{PSL}(2,\mathbb{R})$ (which is conjugated to $\mathrm{PU}(1,1)$). Whether it is one or the other can be read in the fact that the trace of the elements of the group generated by the two matrices are all real. \noindent Whenever we are in one of the following cases 

\begin{itemize}

\item $\mathrm{Im}(\alpha)$ is dense in $\C^*$,

\item $\mathrm{Im}(\alpha)$ is dense in $\R^*$,

\item $\mathrm{Im}(\alpha)$ contains a subgroup dense in $\U$,

\end{itemize}

\noindent we can find a configuration (up to the action of the mapping class group) for which $A$ and $B$ satisfy the hypothesis required to apply Jorgensen's lemma to find that the closure of the image of the Torelli group in $ \mathrm{PGL}(\Ht) \simeq  \mathrm{PSL}(2,\C) $ is 

\begin{itemize}

\item the whole $ \mathrm{PGL}(\Ht) $ when $\mathrm{Im}(\alpha)$ is dense in $\C^*$;

\item the stabiliser in $ \mathrm{PGL}(\Ht)$ of the projectivisation of the subset of elements of $\Ht$ which are totally real when    $\mathrm{Im}(\alpha)$ is dense in $\R^*$;

\item either, if $\mathrm{Im}(\alpha)$ is contained and dense in $\mathbb{U}$, the projectivisation of the group of isometries of the Euclidean volume form (which has signature $(1,1)$); or the whole $ \mathrm{PGL}(\Ht) $ when $\mathrm{Im}(\alpha)$ is contains a subgroup dense in $\U$ and $\mathrm{Im}(|\alpha|)$ is a discrete subgroup of $\R^*_+$.

\end{itemize}

The last case can be refined. Since the only thing we need is to find $A$ and $B$ in $\mathrm{Im}(\alpha)$ such that $(2 - A - A^{-1})(2 - B - B^{-1}) < 1$, the image of the Torelli group is dense when $\mathrm{Im}(\alpha) = \langle \exp(\frac{2i\pi}{n} \rangle$ for $n \neq 1, 2, 3, 4, 6$. This comes from the fact that $(2 -\exp(\frac{2ik\pi}{n}  - \exp(\frac{2ik\pi}{n} ^{-1}) = 2(1 - \cos(\frac{2ik\pi}{n}))$ which can be strictly smaller than $1$ if and only if $n$ is not in the list $1, 2, 3, 4, 6$. In fact, when $n$ belong to this list, the image of $\Stab(\alpha)$ by the Chueshev representation is discrete since it lies in $\mathrm{PGL}(2, \Z)$ if $n=1$ or $2$, $\mathrm{PGL}(2, \Z[i])$ if $n=4$ and $\mathrm{PGL}(2, \Z[\omega])$ if $n=3$ or $6$.We also ask the question whether these images have finite index in the arithmetic groups they are contained in.

Working out the remaining cases requires to also make use of elements of $\Stab(\alpha)$ which are not in the Torelli group. So far it is unknown if the image of the Torelli group can be discrete when $\mathrm{Im}(\alpha)$ is a discrete cyclic subgroup of $\R^*$.

\paragraph{\bf When separating curves do not suffice.}
\label{special}

For some $\alpha$ such that $\mathrm{Im}(\alpha)$ is a discrete cyclic subgroup of $\R^*$, any pair of Dehn twists as in the preceding paragraph fail to generate a non-discrete group. To deal with this case, we are going to appeal to elements of $\Stab(\alpha)$ which do not belong to $\tor$. Assume that $a$ is a simple closed curve which belongs to $\mathrm{Ker}(\delta)$. Then $T_{a}$  the Dehn twist along $a$ belongs to $\Stab(\alpha)$. Let $\delta$ be a separating curve such that $\delta$ and $\nu$ are respectively the blue and red curves on Figure \ref{genustwo}.

\begin{figure}[!h]
\centering
\includegraphics[scale=0.7]{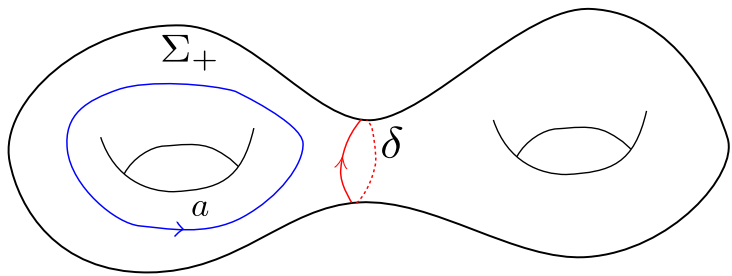}
\caption{The curves $a$ and $\delta$.}
\label{genustwo}
\end{figure}

\begin{lemma}
\label{particular}
The group generated by the action of $T_{\delta}$ and a twist $T_a$ with $a \in \mathrm{Ker}(\alpha)$ and $a \subset \Sigma_+$  generate a parabolic subgroup of $ \mathrm{PGL}(\Ht) $ whose matrices in a base $\mu_{\delta}$, $\lambda$ where $\lambda$ is such that $\lambda(\delta)=1$ are  $$ \begin{pmatrix}

1 & 1 \\
0 & 1
\end{pmatrix}, \begin{pmatrix}

1 &  1 - \alpha(b) \\
0& 1
\end{pmatrix} $$

\noindent where $b$ is a simple closed curve of $\Sigma_+$ such that $a \wedge b = 1$.
\end{lemma}

\begin{proof}

This lemma is the consequence of two remarks:

\begin{itemize}
\item $\mu_{\delta}$ and $\mu_{a}$ are equal. It suffices to compute their values on an appropriate basis, and from the fact that $\alpha(a) = 1$ follows the equality.

\item Choosing appropriate lifts $\tilde{a}, \tilde{b}$ and $\tilde{\delta}$ to $\pi_1 \Sigma$, we have $\tilde{\delta} = [\tilde{a}, \tilde{b}]$. Then for any $\lambda \in \mathrm{Z}^1_{\alpha}(\Gamma, \mathbb{C})$, $\lambda(\tilde{\delta}) =  (1- \alpha(b))\lambda(\tilde{a}) -  (1- \alpha(a))\lambda(\tilde{b})  =  (1- \alpha(b))\lambda(\tilde{a})$. 

\end{itemize}

We deduce the lemma from Proposition \ref{action}.
\end{proof}

The important corollary of this lemma is that if $\mathrm{Im}(|\alpha|)$ has discrete non-trivial image in $\R^*_+$,  $r_{\alpha}(\Stab(\alpha))$ has dense image in $\mathrm{PGL}_{\R}(\Ht) \simeq \mathrm{PGL}(2,\R)$ if $\mathrm{Im}(\alpha) \subset \R^*$.

%
%
%
%
%

\subsection{Arbitrary genus.}

In this subsection, we do complete the proof of Theorem \ref{image}. The idea to deal with arbitrary genus is to see that embedded genus two surfaces give rise to copies of $\mathrm{PGL}(2,\R), \mathrm{PGL}(2,\C) $ or $\mathrm{PU}(1,1)$ in $\mathrm{PGL}(\Ht) \simeq \mathrm{PGL}(2g-2, \C)$ (depending on whether $\alpha$ is totally real, unitary or generic) and that those many embeddings allow to generate the whole $\mathrm{PGL}(2g-2,\R), \mathrm{PGL}(2g-2,\C) $ or $\mathrm{PU}(g-1,g-1)$, still depending on the nature of $\alpha$ (except for the finite number of exceptional unitary cases). To this purpose we will need a lemma of representation of Lie groups that we expose in the next paragraph.

\begin{lemma}
\label{racine}
Let $G < \mathrm{SL}(n,\C)$ be an irreducible connected semi-simple Lie group containing a copy of $\mathrm{SL}(2,\C)$ preserving a decomposition $\C^2 \oplus \C^{n-2}$, acting naturally on $\C^2$ and fixing $\C^{n-2}$. Then $G = \mathrm{SL}(n,\C)$ or $\mathrm{Sp}(n,\C)$.

\end{lemma}

Its proof can be found in \cite[Proposition 6.4, p345]{BeuckersHeckman}. It is actually stated in a more general form in the article and the version we present is a specialisation of their statement for $c= +1$. We will also need the two following geometric lemmas:

\begin{lemma}
\label{irreducible}
For any $\alpha \in \Hl$, $r_{\alpha}(\mathrm{Stab}(\alpha)) \subset \mathrm{GL}(\Ht)$ acts irreducibly on $\Ht$.
\end{lemma}

\begin{proof}
Let $V$ a strict subvector space of $\Ht$. There exists a free homotopy class $\delta$ of null-homologous curves, non trivial homotopically such that:

\begin{itemize}

\item $\exists \varphi \in V$ such that $\varphi([\delta]) \neq 0$;\sk

\item $\mu_{\delta}$ (see Section \ref{remarkable}) does not belong to $V$. 

\end{itemize}

$T_{\delta}$ carries $\varphi$ outside $V$ since $T_{\delta}\cdot \varphi = \varphi + \varphi([\delta]) \mu_{\delta}$. This implies that no strict subvector space of $\Ht$ is $\mathrm{Stab}(\alpha)$-invariant. 
\end{proof}

\begin{lemma}
\label{2dim}
\begin{enumerate}

\item  If $\mathrm{Im}(\alpha) \subset \mathbb{U} $ and is different from $\langle \exp(\frac{2i\pi}{n}) \rangle$ for $n= 1,2,3,4,6$, then the closure of $r_{\alpha}(\mathrm{Stab}(\alpha)) \subset \mathrm{PGL}(\Ht) \simeq \mathrm{PGL}(2g-2, \mathbb{C})$ contains a copy of $\mathrm{PU}(1,1)$ whose lift to $\mathrm{GL}(2g-2, \mathbb{C})$ preserves a decomposition $\C^2 \oplus \C^{2g-4} $, acting as $\mathrm{U}(1,1)$ on the $\C^2$ factor and trivially on the other.

\item  If $\mathrm{Im}(\alpha) \subset \mathbb{R} $ and $\alpha$ is not unitary, then the closure of $r_{\alpha}(\mathrm{Stab}(\alpha)) \subset \mathrm{PGL}(\Ht) \simeq \mathrm{PGL}(2g-2, \mathbb{C})$ contains a copy of $\mathrm{PSL}(2,\R)$ whose lift to $\mathrm{GL}(2g-2, \mathbb{C})$ preserves a decomposition $\C^2 \oplus \C^{2g-4} $, acting as $\mathrm{SL}(2,\R)$ on the $\C^2$ factor and trivially on the other.

\item   If $ \mathrm{Im}(\alpha)$ is any other subgroup of $\C^*$, then the closure of $r_{\alpha}(\mathrm{Stab}(\alpha)) \subset \mathrm{PGL}(\Ht) \simeq \mathrm{PGL}(2g-2, \mathbb{C})$ contains a copy of $\mathrm{PSL}(2,\C)$ whose lift to $\mathrm{GL}(2g-2, \mathbb{C})$ preserves a decomposition $\C^2 \oplus \C^{2g-4} $, acting as $\mathrm{SL}(2,\C)$ on the $\C^2$ factor and trivially on the other.

\end{enumerate}

\end{lemma}

\begin{proof}

The strategy is to choose an appropriate couple of Dehn twists whose action generate the wanted subgroup. For a generic $\alpha$, a couple of curves $\delta$ and $\nu$ like in the figure below will do the trick.

\begin{figure}[!h]
\centering
\includegraphics[scale=0.5]{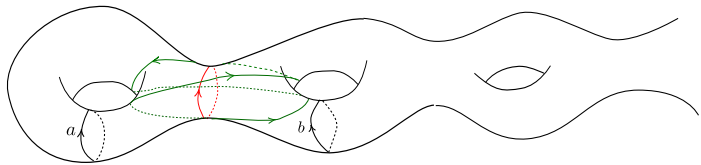}
\caption{Two Dehn twists along the curves $\delta$ and $\nu$.}
\label{twotwists}
\end{figure}

The evil cases are when we cannot choose $a$ and $b$ such that $(2 - A - A^{-1})(2 - B - B^{-1}) > 1$, where $A = \alpha(a)$ and $B = \alpha(B)$, see subsection \ref{2dehntwists}. These corresponds to cases such that either 

\begin{itemize}

\item $\alpha$ is unitary with image $\langle \exp(\frac{2i\pi}{n}) \rangle$ for $n= 1,2,3,4,6$;

\item $\mathrm{Im}(|\alpha|)$ is discrete.

\end{itemize}

There is nothing that can be done in the first case, for we have seen that when $\mathrm{Im}(\alpha) = \langle \exp(\frac{2i\pi}{n}) \rangle$ for $n= 1,2,3,4,6$ the image of the $\mathrm{Stab}(\alpha)$ representation we are considering is a discrete subgroup of $\mathrm{PGL}(\Ht)$. 
\\ In the second case, the trick we used in the genus $2$ case can be recycled: using Dehn twists along simple closed curves of $\mathrm{Ker}(\alpha)$ (see paragraph \ref{special}). 

Except for the unitary exceptional cases, this leads to generating a dense subgroup in $\mathrm{U}(1,1)$, $\mathrm{SL}(2,\R)$ or $\mathrm{SL}(2,\C)$ preserving the plane generated by $\mu_{\delta}$ and $\mu_{\nu}$. One can easily find a set of curves $(\eta_i)_{1\leq i \leq 2g-4}$ such that 

\begin{itemize}

\item the $\eta_i$'s are homologically trivial;

\item the  $\eta_i$'s are all disjoint from $\delta$ and $\nu$;

\item the $\mu_{\delta_i}$ complete $\mu_{\delta}$ and $\mu_{\nu}$ in a basis of $\Ht$.

\end{itemize}

The subspace generated by the $\eta_i$'s is therefore fixed by the action of the subgroup we have constructed and this finishes the proof of the lemma.
\end{proof}

We are now ready to complete the proof of Theorem \ref{image}. Let $G$ be the connected component of the closure of the image of $\Stab(\alpha)$ in $\mathrm{PGL}(\Ht) \simeq \mathrm{PGL}(2g-2, \C)$; it is a real Lie subgroup of $ \mathrm{PGL}(2g-2, \C)$. We prove the theorem depending on whether $\alpha$ is unitary, real or generic.

\begin{enumerate}

\item \textit{$\alpha$ is generic}

\vspace{2mm} \noindent Consider the subgroup generated by couples of Dehn twists like in the proof of Lemma \ref{2dim}. The closure of such a subgroup belongs to $G$ and the latter therefore satisfies the conclusion of Lemma \ref{2dim}, and is also a connected complex Lie subgroup. Also Lemma \ref{irreducible}  ensures that the action of (a lift of) $G$ on $\C^n$ is irreducible. Lemma \ref{racine} implies that $G'$ is therefore the whole $\mathrm{PGL}(2g-2,\C)$, since $G$ does not preserve any symplectic form. 

\item \textit{$\alpha$ is totally real and not unitary}

By complexifying, we are going to be reduced to the previous case. If $G$ was a strict subgroup of $\mathrm{PGL}(2g-2,\R)$, $G_{\C}$ its complexification sitting in $\mathrm{PGL}(2g-2,\C)$ would also be a strict subgroup. But then (a lift of) $G_{\C}$ enjoys the following properties: 

\begin{itemize}

\item its action is irreducible;

\item it contains a copy of $\mathrm{SL}(2, \C)$ stabilising a projective line.

\end{itemize}

\noindent The reasoning of the previous case therefore applies and we can conclude that $G$ is the whole $\mathrm{SL}(2g-2,\R)$.

\item \textit{$\alpha$ is unitary and not exceptional}

\noindent $\mathrm{SU}(g-1,g-1)$ is an other real form of $\mathrm{SL}(2g-2, \C)$, the complexification operation carried out in the totally real case generalises and we are able to conclude that $G$ is the whole $\mathrm{PU}(g-1,g-1)$.

\end{enumerate}

\section{Geometrising representations.}

Consider $\rho : \Gamma \longrightarrow \Aff$ a representation and denote by $\alpha \in \Hl$ its linear part. We are going to deal with the question of whether $\rho$ is the holonomy of a branched affine structure depending on the orbit of $\alpha$ under the $\Mod$-action, which has been studied in Section \ref{modlinear}. We indicate how we are going to proceed: first we deal with the linear parts and prove that every linear character $\alpha\in \Hl$ is realised as the linear part of the holonomy of a branched structure. Then we remark that for a given character $\alpha\in \Hl$, the set of geometric elements(that is, realised by a branched affine structure) of $\mathbf{P}(\Ht)$ is open and invariant by the action $\Stab(\alpha)$. This remark, together with the explicit description of the closure of the image of $\Stab(\alpha)$ in $\mathrm{PGL}(\Ht)$ obtained in the previous section, will allow us to deal with almost all cases. The remaining one, namely when $\mathrm{Im}(\alpha)$ is unitary and finite of order $2,3,4$ or $6$, will be dealt with in a more geometric way, giving explicit models realising admissible representations by means of simple surgeries on translation surfaces. We prove the following theorem:

\begin{thm}

Let $\rho : \Gamma \longrightarrow \Aff$ be a non-abelian representation.

\begin{itemize}

\item If $\rho$ is not Euclidean, then it is the holonomy of a branched affine structure.

\item If $\rho$ is Euclidean, it is the holonomy of a branched affine structure if and only if its volume is positive. 

\end{itemize}

\end{thm}

\subsection{Abelian representations.}
\label{abelian}
We say a word about Abelian representations which have already essentially been dealt with  by Haupt. Actually, Haupt only handles the case where the representations are translations. But since a affine structure whose holonomy is Abelian is either a translation surface or the exponential of a translation surface, the general case of Abelian surface can easily been deduced from Haupt theorem.

\begin{prop}
\label{abelianrep}
Assume that $(\Sigma, \A)$ is a  branched affine surface whose holonomy is abelian, and that $\Sigma$ is not a translation surface (assumption which is equivalent to the statement that the image of the holonomy falls in the subgroup of transformation fixing a given point in $\C$). Then there exists a translation structure $\mathcal{T}$ on $\Sigma$ such that $\A = \exp(\mathcal{T})$. 

\end{prop}

\proof Suppose $f$ is a developing map of $\A$ for which the corresponding holonomy fixes $0 \in \C$. Then $\omega = \frac{df}{f}$ defines a meromorphic differential on the underlying Riemann surface. The poles of $\omega$ are exactly at the points at which $f$ cancels and its residue at such a  point is the order at which $f$ cancels at this point. The residue formulae implies that the sum of the residues must equal $0$, which implies that there are no poles (in particular $f$ does not cancels). Thus $\omega$ is holomorphic and defines a translation structure $\mathcal{T}$ on $\Sigma$ whose exponential if $\A$, for $\omega = \frac{df}{f}$ is the differential of (the locally well-defined) logarithm of $f$.

\subsection{Linear parts.}

The set of linear parts which are not geometric is a closed $\Mod$-invariant subset of $\Hl$. It is in particular the union of the orbit closures of the elements it contains. Based on this remark we are going to prove the

\begin{prop}
\label{linearr}
Every element in $\Hl$ is the linear part of a branched affine structure.

\end{prop}

\begin{proof}

To prove the proposition, we are going to prove that each possible orbit closure of the $\Mod$-action contains a least one element which is realised by a branched affine structure (which is sufficient since the set of geometric linear holonomies is open). This task will be made easy by the fact that these orbit closure are almost all characterised by the (closure of the) image of their elements, and we therefore distinguish cases depending on the closure of $\mathrm{Im}(\alpha) \subset \C^*$.

\begin{enumerate}

\item $\mathrm{Im}(\alpha)$ is unitary. We perform the 'adding a handle' surgery on a translation surface of genus $g-1$ along two segments of same length. Assume that the angle in between the two segment is $\theta$. Then the image of the linear part of the new surface of genus $g$ has image $\langle e^{i\theta} \rangle$. We realise this way elements of the closure of the orbit of any unitary representation.

\item $\mathrm{Im}(\alpha)$ is not unitary. In that case there exists a translation  structure $\mathcal{T}$ on $\Sigma$ such that $\exp(\mathcal{T})$ has linear holonomy $\alpha$. That can easily be deduced from the following fact. Let $a$ and $b$ are arbitrary non zero complex numbers which generate a non unitary subgroup of $\C^*$. There exists $a'$ and $b'$ such that $a= \exp(a')$ and $b = \exp(b')$, and $(a,b)$ form an oriented basis of $\C$as a $\R$-vector space. 
\noindent If $a_1, b_1, \cdots, a_g, b_g$ is a  suitable symplectic basis of $\mathrm{H}_1(\Sigma, \Z)$, applying the last fact to the $g$ couples $(\alpha(a_i), \alpha(b_i))$ leads to construct $\alpha' : \mathrm{H}_1(\Sigma, \Z) \longrightarrow \C$ satisfying the hypothesis of Haupt's theorem and such that $\exp(\alpha') = \alpha)$. Geometrising $\alpha'$ by means of Haupt's theorem and exponentiating leads to geometrising $\alpha$. Note that the full holonomy of the structure is abelian and agrees with $\alpha$.

%
%
%

\end{enumerate}

\end{proof}

\begin{rem}

Based on Proposition \ref{openess} it is easy to prove that every strictly affine representation is geometric. For every non unitary $\alpha : \Gamma \longrightarrow \C^*$ there exists a translation structure on $\Sigma$ whose exponential has holonomy $\alpha$. This $\alpha$ represents the point $0$ in $\Ht$ and by Proposition \ref{openess} the set of geometric representations of linear holonomy $\alpha$ is an open set containing $0$, its projection to $\mathbf{P}(\Ht)$ is therefore onto.

\noindent This trick does not work with euclidean representations since the only abelian geometric and unitary representation have trivial linear holonomy; it is a consequence of Proposition \ref{abelianrep} and Haupt's theorem (Theorem \ref{Haupt}). That is for geometrising these representations that the mapping class group dynamics point of view is efficient.

\end{rem}

\subsection{A tiny bit of linear algebra.}

We classify the orbits of the respective actions of $\mathrm{PGL}(2g-2, \R)$ and $\mathrm{PU}(g-1,g-1)$ on $\mathbb{CP}^{2g-3}$.

\begin{prop}
\label{orbit1}
\begin{itemize}

\item If $g=2$, the action of $\mathrm{PGL}(2g-2, \R) = \mathrm{PSL}(2,\R)$ on $\mathbb{CP}^1$ has three orbits: the real line, the upper half plane and the lower half plane.

\item If $g\leq 3$, the $\mathrm{PGL}(2g-2, \R)$ -action on  $\mathbb{CP}^{2g-3}$ has exactly two orbits: the set of real lines and its complement.

\item The action of $\mathrm{PU}(g-1,g-1)$ has exactly three orbits on $\mathbb{CP}^{2g-3}$: the sets of line which are positive, null and negative for the hermitian form preserved by $\mathrm{U}(g-1,g-1)$.

\end{itemize}
\end{prop}

We leave to the reader the proof of this elementary proposition. 

\subsection{Strictly affine representations.}

We are now set to prove that every strictly affine representation is the holonomy of a branched affine structure. Recall that if $\alpha \in \Hl$ is strictly affine(meaning that $|\alpha|$ is non-trivial):

\begin{itemize}

\item the set of representations $\rho : \Gamma \longrightarrow \Aff$ whose linear part is $\alpha$ and which are not abelian, up to the action by conjugation of $\Aff$ is parametrised by the projectivised space $\mathbf{P}(\Ht)$;

\item the image of the action of $\Stab(\alpha)$ on $\mathbf{P}(\Ht)$ is dense in either $\mathrm{PGL}(2g-2, \R)$ or  $\mathrm{PGL}(2g-2, \C)$ depending on whether $\mathrm{Im}(\alpha)$ is included in $\mathbb{R}^*$ or not; 

\item the subset of $\mathbf{P}(\Ht)$ that parametrises geometric representations is open, empty and $\Stab(\alpha)$-invariant.

\end{itemize}

The set of non-realisable representations is therefore a closed subset, invariant by the action of $\Stab(\alpha)$. According to Proposition \ref{orbit1}, such a set is either the whole $\mathbf{P}(\Ht)$, or the set of totally real representations(only in the case where $\alpha$ is real). It cannot be the whole $\mathbf{P}(\Ht)$ because its complement is non-empty. We now show that if $\alpha$ is real, the set of totally real representations contains geometric representations. In subsection \ref{abelian}, we show that the abelian representation defined by $\alpha$ is realised by an affine structure. The Ehresman-Thurston argument implies that there is an open set containing $\alpha$ in $\mathrm{Hom}(\Gamma, \Aff)$ which is realised by affine branched structure. Such an open set must contain non abelian, totally real representation whose linear part is $\alpha$. We have proven the 

\begin{prop}

If $\rho$ is a strictly affine representation, then it is the holonomy of a branched affine structure. 

\end{prop}

\subsection{Euclidean representations.}

If almost all cases of Euclidean representations can be dealt with using the $\Stab(\alpha) < \Mod$, few exceptional cases need a specific treatment. We will make the distinction between the case where $\mathrm{Im}(\alpha)$ is $\langle \exp(\frac{2i\pi}{n}) \rangle$ for $n= 1,2,3,4,6$(\textbf{exceptional cases}) and the other cases (\textbf{generic cases}). 

\subsubsection{Generic case.}

In the case where $\alpha$ is generic, we have seen that (Theorem \ref{image}) $\Stab(\alpha) < \Mod$ acts on $\mathbf{P}(\Ht)$ through a subgroup dense in $\mathrm{PU}(g-1,g-1)$. Since $\mathrm{PU}(g-1,g-1)$ acts transitively on the set of representations of positive volume (thought of as as subset of $\mathbf{P}(\Ht)$),  any representation of positive volume whose linear part is generic is realised by a branched affine structure.

\subsubsection{Explicit realisation of the exceptional cases.}

When $\mathrm{Im}(\alpha)$ is generated by $-1$, $i$, $\omega_3 = \exp(\frac{2i\pi}{3})$ or $\omega_6 = \exp(\frac{i\pi}{3})$, the image of $\mathrm{Stab}(\alpha)$ is discrete in $\mathrm{PGL}(\Ht)$ and the arguments used to deal with every other cases cannot be applied to determine the elements in $\Ht$ that can be realised by a geometric structure. We first prove a lemma giving a normal form for representations whose linear part is finite.

\begin{lemma}
\label{normalform}
Let $\rho : \Gamma \longrightarrow \Aff$ be a representation such that $\mathrm{Im}(\alpha)$ is finite of order $n$, where $\alpha$ is the linear part of $\rho$. Then there exists a generating set $\{A_1, B_1, \cdots, A_g, B_g \}$ of $\Gamma = \pi_1 \Sigma$ with $\prod_{i=1}^g{[A_i,B_i]}=1$ such that \begin{itemize}

\item $\{A_1, B_1, \cdots, A_g, B_g \}$ projects to a symplectic basis of $\mathrm{H}_1(\Sigma, \Z)$;

\item $\rho(A_1) = z \mapsto e^{\frac{2i\pi}{n}}z + k$;

\item $\rho(B_1) = z \mapsto z$;

\item for all $i > 1$, $\rho(A_i)$ and $\rho(B_i)$ are translations.

\end{itemize}

\end{lemma}

\begin{proof}

The proof is elementary. Since $\alpha$ has a finite image, we can find a symplectic basis $a_1, b_1, \cdots, a_g, b_g$ of $\mathrm{H}_1(\Sigma, \Z)$ for which $\alpha(a_1) = e^{\frac{2i\pi}{n}}$ and $\alpha(b_1) = \alpha(a_2) = \cdots = \alpha(b_g) = 1$.
\noindent The basis  $a_1, b_1, \cdots, a_g, b_g$ can be lifted to $\{A_1, B_1, \cdots, A_g, B_g \}$  a generating set of $\pi_1 \Sigma$ such that $[A_1, B_1]^{-1} = \prod_{i=2}^g{[A_g,B_g]}$. Since for $i>2$, $\rho(A_i)$ is a translation, $\rho(\prod_{i=2}^g{[A_g,B_g]}) = \prod_{i=2}^g{[\rho(A_g),\rho(B_g)]}$ must be trivial, hence $\rho(A_1)$ and $\rho(B_1)$ must commute. But since $\rho(A_1)$ is a non-trivial rotation and  $\rho(B_1)$ is a translation, $\rho(B_1)$ must be trivial. Which proves the lemma.

\end{proof}

We now explain the strategy we are going to follow. For the remainder of the section, $\rho$ is a Euclidean representation whose linear part has finite image and positive volume. From the normal form of the representation we deduce an Abelian representation made of translation on a surface $\Sigma'$ of genus $g-1$ which factors through $p : \mathrm{H}_1(\Sigma', \Z) \longrightarrow \C$. If we can find a translation structure on $\Sigma'$ whose period morphism is $p$, we can geometrise $\rho$: it is enough to apply the surgery presented in Section \ref{adding} (the 'adding a handle' surgery) to the translation structure along two segments forming an angle of $\frac{2i\pi}{n}$ to get a surface whose holonomy is (up to conjugation) $\rho$.  

\vspace{2mm}

According to Lemma \ref{normalform}, $\rho$ splits into two representations of the fundamental group of a surface of genus $g-1$ and of a torus. The image of the representation induced on the torus has volume zero because its image is cyclic. The volume of $\rho$ is therefore equal to the volume of $p \in \mathrm{H}^1(\Sigma', \C)$ induced on $\Sigma'$ of genus $g-1$, hence the volume of $p$ is positive. If $p$ can be realised by a translation surface, we can realise $\rho$ as we previously explained. Recall Haupt's theorem that characterises periods which can be realised by a translation surface:

\begin{thm}[Haupt, \cite{Haupt}]
Let $\omega \in \mathrm{H}^1(\Sigma, \C)$. It can be realised as the period morphism of a translation surface if and only if the two following conditions hold:

\begin{enumerate}

\item the volume of $\omega$ is positive;

\item if the image of $\omega$ is a lattice $\Lambda$ in $\C$, then $\mathrm{vol}(\omega) > \mathrm{vol}(\C / \Lambda)$.

\end{enumerate}

\end{thm}

Since $p$ has positive volume $\rho$ can be realised unless it has discrete image in $\C$ and violates the second condition of Haupt's theorem. In the latter case, we can suppose that its image is $\Z \oplus i \Z$ up to an affine renormalisation. 
Our strategy at this point is to put $\rho$ in a form which will allow us to realise it starting from the flat torus $\C / \Z \oplus i\Z $ and performing successive 'adding a handle surgeries'. 
\noindent According to \cite[Proposition 2.7]{Deroin}, it must have volume $1$ and one can find a symplectic basis of $\Sigma'$, call it $a'_2, b'_2, \cdots, a'_g, b'_g$ such that $p(a'_g) = 1$, $p(b'_g) = 1$ and $p(a'_i) = p(b'_i) = 0$ for $1< i < g$. We deduce from this remark that there exists a presentation of $\Gamma = \{ A'_1, B'_1, \cdots, A'_g, B'_g  \ | \ \prod_{i=1}^g{[A'_g,B'_g]}\}$ such that 
:

\begin{itemize}

\item $\rho(A'_1) = z \mapsto \exp(\frac{2i\pi}{n}) z$;

\item $\rho(B'_1) = z \mapsto z$;

\item $\rho(A'_i) = \rho(B'_i) = z \mapsto z$ for $1 < i < g$;

\item $\rho(A'_g) = z \mapsto z + 1$;

\item $\rho(B'_g) = z \mapsto z + i$.

\end{itemize}

By applying appropriate Dehn twists in the subsurface of genus $g-1$ with one boundary component whose fundamental group is generated by $A'_1, B'_1, \cdots, A'_{g-1}, B'_{g-1}$ we easily deduce a presentation $\Gamma = \{ A''_1, B''_1, \cdots, A''_g, B''_g  \ | \ \prod_{i=1}^g{[A''_g,B''_g]}\}$ such that:

\begin{itemize}

\item $\rho(A'_i) = z \mapsto \exp(\frac{2i\pi}{n}) z$ for $1 \leq i < g$;

\item $\rho(B'_i) = z \mapsto z$ for $1 \leq i < g$;

\item $\rho(A'_g) = z \mapsto z + 1$;

\item $\rho(B'_g) = z \mapsto z + i$.

\end{itemize}

Such a representation is easily realised by starting from the torus $\C / \Z \oplus i \Z$ and realising $g-1$ successive 'adding a handle' surgeries.

\bibliographystyle{alpha}
\bibliography{biblio}

\end{document}